\documentclass[12pt]{article}
\usepackage{amsmath,amsfonts,amssymb,amsthm,amscd}
\title{Embedded Picard-Vessiot extensions}
\date{\today}
\author{Quentin Brouette\\University of Mons \and   Greg Cousins\thanks{Partially supported by NSF grants DMS-1360702 and DMS 1665035}\\University of Notre Dame \and Anand Pillay\thanks{Partially supported by NSF grants DMS-1360702 and DMS-1665035}\\University of Notre Dame  \and Francoise Point \\University of Mons}

\newtheorem{Theorem}{Theorem}[section]

\newtheorem{Definition}[Theorem]{Definition}
\newtheorem{Remark}[Theorem]{Remark}
\newtheorem{Lemma}[Theorem]{Lemma}
\newtheorem{Corollary}[Theorem]{Corollary}
\newtheorem{Fact}[Theorem]{Fact}

\newcommand{\Q}{\mathbb Q}

\begin{document}
\maketitle

\begin{abstract}  We prove that if $T$ is a theory of large, bounded, fields of characteristic $0$ with almost quantifier elimination,  and $T_{D}$ is the model companion of $T\cup \{``\partial$ is a derivation"\}, then for any model $({\cal U},\partial)$ of $T_{D}$, differential subfield $K$ of ${\cal U}$ such that $C_{K}\models T$, and linear differential equation $\partial Y = AY$ over $K$, there is a Picard-Vessiot extension $L$ of $K$ for the equation with $K\leq L \leq {\cal U}$, i.e. $L$ can be embedded in $\cal U$ over $K$, as a differential field.   Moreover such $L$ is unique to isomorphism over $K$ as a differential field. Likewise for the analogue for strongly normal extensions for logarithmic differential equations in the sense of Kolchin. 
\end{abstract}

\section{Introduction and preliminaries}
Recent papers such as \cite{CHvdP} and \cite{KP} have shown that under certain conditions (on the differential field $(K,\partial)$ and its field $C_{K}$ of constants), given a linear differential equation over  $(K,\partial)$ we can find a Picard-Vessiot extension $(L,\partial)$ of $(K,\partial)$ for the equation such that  $C_{K}$ is existentially closed in $L$ (as a field).  Among the motivating examples to which this applies is the case where $C_{K}$ is real closed and $K$ is formally real.  Now there is a certain complete first order theory $CODF$, which is the model companion of the theory of formally real fields equipped with a derivation, whereby $(K,\partial)$ from the previous sentence, will be embedded in a model $({\cal U},\partial)$ of $CODF$. And it is natural to ask whether for any such model $({\cal U},\partial)$ the Picard-Vessiot extensions of $K$ can be found {\em inside} ${\cal U}$ (over $K$)? 
In this paper we prove a general result, namely the result stated in the abstract, which will yield a positive answer.  

In the case where the theory $T$ in the abstract is $ACF_{0}$,  $C_{K}$ is algebraically closed, $\cal U$ will be differentially closed, and the result (that $L$ can be found inside $\cal U$ over $K$) is well-known. The model theoretic account goes via prime models as follows: The prime model $K^{diff}$ of $K$ embeds  in $\cal U$ over $K$, has no new constants, and the linear differential equation has a fundamental system of solutions in $K^{diff}$ as the latter is differentially closed.  In the more general situations such as when $\cal U$ is a model of $CODF$, this approach has no chance of working, as there are no prime models (see \cite{Singer} and also \cite{Point} which adapts Singer's argument to other contexts).  But as  it turns out we are able to  combine the relatively hard abstract existence statements from  \cite{KP} with some relatively soft model theory to obtain the embedded existence statements and this is the content of the current paper.  See also Lemma 4.4 of \cite{BP} and the paragraph following it which discuss related issues. 

In the remainder of this section we give the necessary definitions and background.  Both the model theory and differential algebra in the paper are fairly basic.  
As a rule, $\cal L$ will denote the language of unitary rings, possibly with some additional constant symbols,  and ${\cal L}_{\partial}$ will be $\cal L$ together with a unary function symbol $\partial$.   In general variables $x,y$  range over finite tuples. 

In this paper we will only be concerned with fields (and differential fields) of characteristic $0$, although many notions will make sense in general.

\begin{Definition} \begin{enumerate}
\item A field  $K$ is said to be {\em bounded} if for each $n$, $K$ has only finitely many extensions of degree $n$. This is also known as Serre's property (F).
\item 
 A field $K$ is said to be {\em large} (or {\em ample}) if for any  algebraic variety $V$ which is defined over $K$, is $K$-irreducible, and has a nonsingular $K$-rational point, $V(K)$ is Zariski dense in $V$.
\end{enumerate}
\end{Definition}

\begin{Remark} \begin{enumerate} 
\item  Large fields were introduced by Pop in \cite{Pop} where one can find other characterizations. 
\item 
The class of large fields is elementary in the language $\cal L$.
\item 
 If $V$ is a $K$-irreducible variety over $K$ with a nonsingular $K$-point then $V$ is absolutely irreducble.
\item
 Boundedness of a field is preserved under elementary equivalence in the language $\cal L$.
\item Suppose $k$ is a field, and $V$ is a $k$-irreducible variety over $k$. Then $k$ is existentially closed in the function field $k(V)$ of $k$ iff $V(k)$ is Zariski-dense in $V$. 
\end{enumerate}
\end{Remark}
\noindent
{\em Explanation.}   2 is Remark 1.3 of \cite{Pop}.  For 3,  note that a nonsingular point on a variety $V$ lies on exactly one (absolutely) irreducible component. So if $V$ is defined over $K$ and $a\in V(K)$ is smooth then the (absolutely)  irreducible component of $V$ on which $a$ lies is defined over $K$, which suffices. 4 is folklore and 5 is a tautology.

\begin{Definition} Let $T$ be a theory of fields in the language $\cal L$. We say that $T$ has {\em almost quantifier-elimination} if whenever $K\models T$, $A\subseteq K$ is relatively algebraically closed in $K$ ( in the field-theoretic sense), $\bar a$ is an enumeration of $A$, and $p(\bar x)$ is the quantifier-free type of $\bar a$, then $T\cup p(\bar x)$ determines a complete type. 
\end{Definition}

\begin{Remark} \begin{enumerate}
\item  Clearly ``$T$ has QE"  implies ``$T$ has almost $QE$"  implies ``$T$ is model-complete".
\item
It is routine to prove that $T$ has almost quantifier elimination if and only if every formula $\phi(x)$ is equivalent modulo $T$ to a formula of the form $\exists y(\psi(x,y))$ where $\psi(x,y)$ is positive quantifier-free and $ACF$ implies that the projection map from $(x,y) \to x$ is finite-to-one on solutions of $\psi(x,y)$. 
\end{enumerate}
\end{Remark}

It is important to note that the bulk of the ``nice"  theories of fields from the point of view of logic, have almost quantifier elimination as well as the property that all models are large and bounded: this is the case for example for $RCF$ (in the field language), $Th(\Q_{p})$, and the theory of pseudofinite fields (with additional constants/parameters).  

Recall that if $T$ and $T'$ are theories in a given language, $T'$ is said to be a model companion of $T$ if $T'$ is model complete and $T$ and $T'$ have the same universal consequences, equivalently if the models of $T'$ are precisely the existentially closed models of $T_{\forall}$.

We now pass to differential fields, which we view as structures 
\newline
$(K,+,\times, -, 0, 1,\partial)$ in the language 
${\cal L}_{\partial}$ where $(K,+,\times)$ is a field and $\partial:K \to K$ is a derivation.  The theory of differential fields has a model companion $DCF_{0}$ which moreover has quantifier elimination.  We can extend the notion of almost quantifier elimination to  differential fields as follows:

\begin{Definition} Let $T$ be a theory  of differential fields in language ${\cal L}_{\partial}$. We say that $T$ has almost quantifier elimination if,  whenever $(K,\partial)$ is a model of  $T$, $A$ is a subfield which is both closed under $\partial$ as well as being relatively algebraically closed in $K$ as a field, and the tuple ${\bar a}$ enumerates $A$ and $p({\bar x}) = qftp({\bar a})$ (the quantifier-free type of $\bar a$) then $T\cup p({\bar x})$ axiomatizes a complete type. 

\end{Definition} 

Remember that, for an irreducible affine variety $V$ over a differential field $(K,\partial)$, the variety $T_{\partial}(V)$ is the variety defined by equations:
$P(x_{1},\ldots,x_{n}) = 0$ and $\Sigma_{i=1,\ldots,n} (\partial P/\partial x_{i})u_{i} + P^{\partial}$ for $P(x_{1},\ldots,x_{n})$ in $I_{K}(V)$, where $P^{\partial}$ is the result of applying the derivation to the coefficients of $P$. 

One of the reasons for the importance of the property of largeness in the current paper is Tressl's uniform model companion for $T\cup\{``\partial$ is a derivation"\} when $T$ is a model-complete theory of large fields, see \cite{Tressl}.  In this  paper of Tressl,  which dealt with the general case of several commuting derivations, the axioms were rather complicated. In the case of a single derivation we can modify slightly the  ``geometric axioms" for $DCF_{0}$ to obtain a more accessible account. (Similar things were done  in \cite{MR} for $CODF$ and more generally   in \cite{Guzy-Point} for differential topological fields.)

\begin{Lemma} Let $T$ be a model complete theory of large fields. Then $T\cup\{``\partial$ is a derivation"\} has a model companion which we call $T_{D}$. Moreover $T_{D}$ can be axiomatized by $T$ together with the following schema:  whenever $V$ is an irreducible affine variety over $K$ with a nonsingular $K$-point, $s:V\to T_{\partial}(V)$ is a $K$-rational section of the natural projection, and $U$ is a Zariski open subset of $V$ defined over $K$  then there is $a\in U(K)$ such that $s(a) = \partial(a)$.
\end{Lemma} 
\begin{proof} 
Strictly speaking we mean $s(a) = (a,\partial(a))$ but here and subsequently we may identify $s(a)$ with the second coordinate. 

We give a sketch proof of the lemma.  Let $\Sigma$ be the collection of axioms stated in the lemma.  We show that the models of $\Sigma$ are precisely the existentially closed models of $(T\cup\{``\partial$ is a derivation"$\})_{\forall} = T_{\forall} \cup \{``\partial$ is a derivation"\}.   Let $(K,\partial)$ be such an existentially closed model. Note first that $K$ must be a model of $T$, because $T$ is $\forall\exists$ axiomatizable and any derivation on a given field extends to a derivation on any larger field. 

 Let  $V$ be an irreducible $K$-variety with a nonsingular $K$-point, and $s: V \to T_{\partial}(V)$ a $K$ rational section of the projection.  Let $a$ be a generic point of $V$ over $K$ (in some ambient algebraically closed field containing $K$). Then as in \cite{Pierce-Pillay}, defining $\partial(a)$ to be $s(a)$ yields an extension of the derivation $\partial$  on $K$ to a derivation, also called $\partial$ of $K(a)$.  On the other hand, as $K$ is large, our assumptions imply that $K$ is existentially closed in $K(a)$ as fields, whereby for some field $L$ extending $K(a)$, $K \prec L$ as fields. Extend the derivation $\partial$ on $K(a)$ to a derivation $\partial$ on $L$. So $(L,\partial)$ is a model of $T\cup\{``\partial$ is a derivation"\}. As $(K,\partial)$ is existentially closed, for any Zariski open $U$ of $V$ over $K$ there is $a_{1}\in U(K)$ such that $s(a_{1}) = \partial(a_{1})$ as required. 

We leave it to the reader to show conversely that any model of $\Sigma$ is an existentially closed model of $T_{\forall} \cup \{``\partial$ is a derivation"\}.  (See \cite{Pierce-Pillay}.)

\end{proof}

A nice application of the axioms is the following:
\begin{Corollary} Let $T$ and $T_{D}$ be as in Lemma 1.6. Let  $(K,\partial)$ be a model of $T_{D}$, and $C_{K}$ its field of constants. Then $C_{K}$ is also a model of $T$ (hence an elementary substructure of $K$ as fields). 
\end{Corollary}
\begin{proof} It suffices to  show that $C_{K}$ is existentially closed in $K$ as a field.  Let $a$ be a tuple from $K$. We have to show any quantifier-free $L$-formula over $C_{K}$ which is true of $a$ is satisfied in $C_{K}$.   Let $V = V(a/C_{K})$ be the variety over $C_{K}$ whose generic point is $a$. As $C_{K}$ is relatively algebraically closed in $K$, $V$ is  absolutely irreducible.   By definition $a$ is a smooth point of $V$. On the other hand $T_{\partial}(V) = T(V)$ the tangent bundle of $V$, and we have the $0$-section $s_{0}:V\to T(V)$ (defined over $C_{K}$).  So for any Zariski open subset $U$ of $V$ defined over $C_{K}$, the axioms give us  $a_{1}$ in $U(K)$ such that $\partial(a_{1}) = 0$, namely $a_{1}\in U(C_{K})$. This suffices. 
\end{proof}

\begin{Remark}  Singer's theory $CODF$  introduced in \cite{Singer} is, on the face of it,  a theory in the language of ordered differential fields, but it is easy to see that it coincides with the expansion of $RCF_{D}$ by the ordering defined by $\exists z(y-x = z^{2})$. 
\end{Remark}

\medskip
We now pass to differential Galois theory. We recommend the survey paper \cite{L-S-P} as a reference (especially as the notation is similar). 
By a linear differential equation in vector form over a differential field $(K,\partial)$ we mean something of the form $\partial Y = AY$ where $Y$ is a column vector of unknowns of length $n$ and $A$ is an $n\times n$ matrix over $K$. A fundamental system of solutions of this equation in a differential field $L$ extending $K$ is by definition a set of solutions $Y_{1},.., Y_{n}$ with coefficients from $L$ which is linearly independent over $C_{L}$. This equivalent to the $n\times n$ matrix whose columns are $Y_{1},..,Y_{n}$, being nonsingular (i.e. nonzero determinant).  So a fundamental system is precisely a solution to $\partial Z = AZ$ where $Z$ is an unknown $n\times n$ matrix in $GL_{n}$.

A {\em  Picard-Vessiot extension} of $K$ for the equation is by definition a differential field extension $L$ of $K$ which is generated over $K$ by a fundamental system of solutions, and such that $C_{L} = C_{K}$.

A  generalization of linear DE's and the Picard-Vessiot theory is Kolchin's strongly normal theory (appearing in the book \cite{Kolchin} for example). The group $GL_{n}$ is replaced by an arbitrary connected algebraic group $G$ over the constants $C_{K}$ of a differential field $K$.  The equation $\partial Z = AZ$ on $GL_{n}$ is replaced by $\partial z\cdot z^{-1} = a$, where $z$ ranges over $G$, $a\in LG(K)$,  and the product $\partial z \cdot z^{-1}$ is in the sense of the tangent bundle $TG$ of $G$ (also an algebraic group). Here $LG$ is the Lie algebra of $G$. When $G = GL_{n}$, $\partial z\cdot z^{-1}$ is precisely the product $(\partial Z) Z^{-1}$ of $n\times n$ matrices, so an equation $\partial z\cdot z^{-1} = A$  is precisely $\partial Z  = AZ$. 

In any case we write  $dlog_{G}(z)$ for the  map from $G$ to its Lie algebra, taking $z$ to $\partial z \cdot z^{-1}$. A  {\em strongly normal extension} of $K$ for a logarithmic differential equation $dlog_{G}(z) = a$ on $G$ over $K$ is by definition a differential field extension $L$ of $K$ generated over $K$ by a solution $g\in G(L)$ of the equation and with no new constants.  So when $G = GL_n$ this is precisely a Picard-Vessiot extension.

When $C_{K}$ is algebraically closed, it is well-known that strongly normal extensions of $K$ (for a given logarithmic differential equation over $K$) exist and are unique up to isomorphism over $K$ as differential fields.

Building on and generalizing work in the Picard-Vessiot case (\cite{GGO}, \cite{CHvdP}),  the following was proved in \cite{KP}.
\begin{Fact} Suppose that $K$ is a differential field, $G$ a connected algebraic group over $C_{K}$, and 
\begin{center}
 $dlog_{G}(z) = a$  (*)
\end{center}
is a logarithmic differential equation on $G$ over $K$. Then
\begin{enumerate}
\item  Suppose that $C_{K}$ is existentially closed in $K$ as fields. Then there exists a strongly normal extension of $K$ for (*). 
\item Suppose in addition that $C_{K}$ is large and bounded. Then there is a strongly normal extension $L$ of $K$ for (*) such that $C_{K}$ is existentially closed in $L$ as fields.
\item Suppose in the context of 2 that $L_{1}$, $L_{2}$ are strongly normal extensions of $K$ for (*) and that there are field embeddings over $K$ of $L_{1}$, $L_{2}$ respectively into a field $L$ such that $C_{K}$ is existentially closed in $L$. Then $L_{1}$ and $L_{2}$ are isomorphic over $K$ as differential fields. 
\end{enumerate}
\end{Fact}

\begin{Remark} \begin{enumerate}\item Let $k$ be any field (of characteristic $0$). Noting that $k$ is existentially closed in the field $k(x)$ of rational functions over $k$, it follows that 1 above applies to the differential field $(K,d/dx)$, where $K = k(x)$. 
\item As pointed out to us by Omar Leon-Sanchez, in Fact 1.9, 2 and 3 above we can  drop the assumption that $C_{K}$ is large when dealing with linear differential equations and Picard-Vessiot extensions, basically because the set of $k$-points of a connected {\em linear}  algebraic group over $k$ is always Zariski-dense.
\end{enumerate}
\end{Remark}

\section{Main results}
In this section we will prove the main theorem of the paper:
\begin{Theorem} Let $T$ be a theory of large, bounded fields with almost quantifier elimination (in the language $\cal L$ of unitary rings possibly with constants).  Let $({\cal U}, \partial)$ be a model of $T_{D}$, and 
let $K$ be a differential subfield of $\cal U$, such that the field $C_{K}$ of constants of $K$ is a model of $T$. 
Let $dlog_{G}(z) = a$ be a logarithmic differential equation over $K$ (with respect to a connected algebraic group $G$ over $C_{K}$). 
 Then we can find a strongly normal extension $L$ of $K$ for the equation which is a differential subfield of $\cal U$. Moreover any two such $L$'s are isomorphic over $K$ as differential fields. 
\end{Theorem}

\begin{Remark} \begin{enumerate} \item The fact that the  equation $dlog_{G}(z) = a$ has a solution in $G(\cal U)$ is an immediate consequence of the axioms in Lemma 1.6. 
The main point of Theorem 2.1  is that there is a solution $g$ in ${\cal U}$ such that $K(g)$ has no new constants. 
\item Note that a special case of the theorem is when $T = ACF_{0}$ in which case $T_{D} = DCF_{0}$. But as mentioned in the introduction this is known directly.
\item  Let us mention roles played by the various hypotheses in Theorem 2.1. Largeness and boundedness are the assumptions on the field of constants in Fact 1.9, 2, which yield  a strongly normal extension $L$  of $K$ such that $C_{K}$ is existentially closed in $L$  Largeness is also needed for the existence of the model companion $T_{D}$. Almost quantifier elimination of $T$ is used (in Lemma 2.3 below) to obtain almost quantifier elimination of $T_{D}$, which after replacing $K$ by its relative algebraic closure inside $\cal U$, allows us to find $L$ inside $\cal U$. 
\end{enumerate}
\end{Remark}

The following lemma will be an important ingredient.
\begin{Lemma} Suppose that $T$ is a theory of large fields and has almost quantifier elimination. Then  $T_{D}$ has almost quantifier elimination (see Definition 1.5)
\end{Lemma}
\begin{proof} Let ${\bar a}$ be a infinite tuple in a model $K$ of $T_{D}$ which enumerates a relatively algebraically closed (in the field sense) differential subfield of $K$, and let $p({\bar x})$ be the quantifier-free type of ${\bar a}$. We show that $p({\bar x})$ axiomatizes a complete type modulo $T_{D}$ by a standard back-and-forth argument inside saturated models. 

So let $K_{1}$ and $K_{2}$ be saturated models of $T_{D}$ and ${\bar b}$, ${\bar c}$ realizations of $p({\bar x})$ in $K_{1}, K_{2}$ respectively.
As $T$ has almost quantifier-elimination and the $\cal L$-reducts of $K_{1}$, $K_{2}$ are models of $T$ it follows that 
\newline
(*) ${\bar b}$ and ${\bar c}$ have the same $\cal L$-type, and moreover each  is relatively algebraically closed in $K_{1}$, respectively, $K_{2}$. 

\vspace{2mm}
\noindent
Now let $d$ be an element of $K_{1}$ and let ${\bar d}$ be an enumeration of  the relative (field-theoretic) algebraic closure in $K_{1}$ of the differential field generated by ${\bar b}$ and $d$.  For the back-and-forth argument to work it will suffice (by symmetry) to find ${\bar e}$ in $K_{2}$ such that the partial $L_{\partial}$-isomorphism taking ${\bar b}$ to ${\bar c}$ extends to one taking ${\bar d}$ to ${\bar e}$.  And for this it will be enough (by saturation of $K_{2}$) to realize in $K_{2}$ any finite part of the copy over ${\bar c}$ of 
the quantifier-free ${\cal L}_{\partial}$-type of ${\bar d}$ over ${\bar b}$. 

Hence we have reduced the argument to showing the following (where $d$ has now a different meaning):
\newline
{\em Claim.}  Let $\phi(x)$ be a quantifier-free ${\cal L}_{\partial}$ formula over ${\bar b}$ which is realized in $K_{1}$ by a finite tuple $d$.  Then the copy of this formula over $\bar c$ is realized in $K_{2}$. 
\newline
{\em Proof of claim.}  This is an adaptation to the current context of a well-known argument (see the proof of Proposition  5.6 in \cite{Pillay-Polkowska}).
The formula $\phi(x)$ is of the form $\psi(x,\partial x,...,\partial^{(r)} x)$ for some $r$ and quantifier-free $\cal L$-formula $\psi$ over ${\bar b}$.  Let $d_{1} = (d,\partial (d),..,\partial^{(r)}(d))$, and let $e_{1} = \partial(d_{1})$.  Let $V_{1}$ be the algebraic variety over $\bar b$ whose generic point is $d_{1}$. As ${\bar b}$ is relatively algebraically closed in $K_{1}$, $V_{1}$ is absolutely irreducible. Moreover $(d_{1},e_{1})\in T_{\partial}(V_{1})$.  Likewise if $W_{1}$ is the variety over $\bar b$ whose generic point is $(d_{1},e_{1})$, then $W_{1}$ is absolutely irreducible. 
By Corollary 1.7 of \cite{Pierce-Pillay} there is $f_{1}$ rational over ${\bar b},d_{1}, e_{1}$ such that $((d_{1},e_{1}),(e_{1},f_{1}))\in T_{\partial}(W_{1})$.  So we can write $(e_{1},f_{1}) = s_{1}(d_{1},e_{1})$ for some $\bar b$-rational section $s_{1}$ of the projection $\pi: T_{\partial}(W_{1})\to W_{1}$. 

Now from (*)  ${\bar b}$ and ${\bar c}$ have the same ${\cal L}$-type in $K_{1}$, $K_{2}$ respectively.  So without the loss of generality the $\cal L$-elementary map $h:{\bar b}\to {\bar c}$ extends to an isomorphism (of fields) which we also call $h:K_{1}\cong K_{2}$.  We let $V_{2},W_{2},s_{2}, d_{2}, e_{2}, f_{2}$ be the images of $V_{1}$ etc.  under $h$. Then $(d_{2},e_{2})$ is a generic point of $W_{2}$ over $\bar c$, $s_{2}$ is a $\bar c$-rational section of the projection $T_{\partial}(W_{2})\to W_{2}$, and $s_{2}(d_{2},e_{2})  = (e_{2},f_{2})$. Hence the axioms for $T_{D}$ from Lemma 1.6, together with saturation of $K_{2}$, imply that there is a generic point $(d_{3},e_{3})$ of $W_{2}$ over $\bar c$, such that $s_{3}(d_{3},e_{3}) = \partial((d_{3},e_{3}))$. But note that  $s_{3}(d_{3},e_{3})$ is of the form $(e_{3},f_{3})$, which implies that $e_{3} = \partial(d_{3})$. 

The upshot is that the $L$-type of $(d_{3},e_{3})$ over $\bar c$ is the image under $h$ of the $L$-type of $(d_{1},e_{1})$ over $\bar b$. As $\partial(d_{1}) = e_{1}$ and $\partial(d_{3}) = e_{3}$, it follows immediately that the image of $\phi(x)$ under $h$ is realized in $K_{2}$, yielding the claim, as well  as the lemma.

\end{proof} 

\begin{Remark}  Lemma 2.3 could also be obtained using  Theorem 7.2 (iii) of \cite{Tressl}  and Remark 1.4.2 above. Namely assuming $T$ to have almost quantifier elimination, add new relation symbols for the formulas $\exists y(\psi(x,y))$ appearing in Remark 1.4.2, to obtain a definitional expansion $T^{*}$ which has quantifier elimination in the new language ${\cal L}^{*}$.  Then the aforementioned result of Tressl is essentially that $T^{*}$ together with the axioms $\Sigma$ from 1.6 has quantifier elimination in ${\cal L}^{*}_{\partial}$. This translates into saying that $T_{D}$ has almost quantifier elimination, as required.   (See also \cite{Guzy-Point}.)  
\end{Remark}

\vspace{5mm}
\noindent
\begin{proof} [Proof of Theorem 2.1.]

Let $({\cal U},\partial)$ be a model of $T_{D}$, $K$ a differential subfield such that $C_{K}\models T$, and let $dlog_{G}(-) = a$ be a logaritmic differential equation over $K$ (where $G$ is a connected algebraic group over $C_{K}$).
 We want first to find a strongly normal extension $L$  of $K$ for the equation which is contained in ${\cal U}$ (equivalently embeds in ${\cal U}$ over $K$ as a differential field). 

\vspace{2mm}
\noindent
{\em Claim.} We may assume that $K$ is relatively algebraically closed in $\cal U$.
\begin{proof} [Proof of Claim.] Let $K_{1}$ be the algebraic closure  of $K$ in $\cal U$ as a field. It is clear that $K_{1}$ is also a differential subfield of $\cal U$.  Now it is well-known that 
$C_{K_{1}}$ is contained in the algebraic closure of the field $C_{K}$. (If  $a\in C_{K_{1}}$ and $P(x)$ is the minimal polynomial of $a$ over $K$, then by applying $\partial$ to $P(a)$ and using that $a$ is a constant, we see that $P$ has coefficients in $C_{K}$.)  But $C_{K}$ being an elementary substructure of $C_{\cal U}$ implies that $C_{K}$ is algebraically closed in $C_{\cal U}$. Hence we see that $C_{K_{1}} = C_{K}$.  But then a strongly normal extension of $K_{1}$ inside $\cal U$ (for the equation) gives rise to a strongly normal extension of $K$ inside $\cal U$. 
\end{proof}

\vspace{2mm}
\noindent
Now, as $C_{K}$ is a model of $T$,  it is large and bounded. Hence Fact 1.9.2 gives us a strongly normal extension $(L,\partial)$ of $(K,\partial)$ for the equation such that $C_{K}$ is existentially closed in $L$ as fields.  It follows that $(L,\partial)$ is a model of $T_{\forall} \cup \{``\partial$ is a derivation"\}. Hence $(L,\partial)$ extends to a model  $(L_{1},\partial)$ of $T_{D}$. 

Now as $K$ is  relatively algebraically closed in the model $({\cal U},\partial)$ of $T_{D}$,  by Lemma 2.3 it follows that (an enumeration of ) $K$ has the same
$L_{\partial}$-type in  $({\cal U},\partial)$ and $(L_{1},\partial)$.  In other words 
\newline 
(**) the structure $({\cal U},\partial)$ with names for elements of $K$ is elementarily equivalent to the structure $(L_{1},\partial)$ with names for elements of $K$.

Let $L = K(g)$ where $dlog_{G}(g) = a$. Then by Lemma 2.2 of \cite{L-S-P} the quantifier-free $L_{\partial}$ type of $g$ over $K$ is isolated  by a formula $\phi(y)$ say.  (One can also just use the fact that $L$ has to live inside some differential closure of $K$.) Note that if $\alpha$ is a solution of $\phi(y)$ in some differential field extension of $K$, then $K(\alpha)$ is isomorphic to $L$ over $K$ (as differential fields), in particular $K(\alpha)$ is also a strongly normal extension of $K$ for the equation.
But by (**) the formula $\exists y\phi(y)$ over $K$ is true in $\cal U$. So this gives us the required strongly normal extension of $K$ inside $\cal U$. 

\vspace{2mm}
\noindent
The uniqueness part of Theorem 2.1 follows from part 3 of Fact 1.9, as if $L_{1}$ and $L_{2}$ are both strongly normal extensions of $K$ inside $\cal U$, then we already have embeddings (as fields) of $L_{1}$ and $ L_{2}$ over $K$  into a field in which $C_{K}$ is existentially closed. 

\end{proof}

\end{document}